\documentclass[11pt]{amsart}
\usepackage{subfigure}
\usepackage{array}
\usepackage{multirow}
\usepackage{amsmath}
\usepackage{amssymb}
\usepackage{graphicx}
\usepackage{booktabs}
\usepackage{tabularx}
\usepackage{color}
\usepackage[pagewise]{lineno}

\newtheorem{theorem}{Theorem}[section]
\newtheorem{corollary}[theorem]{Corollary}
\newtheorem{lemma}[theorem]{Lemma}

\theoremstyle{definition}
\newtheorem{definition}[theorem]{Definition}

\newtheorem{remark}[theorem]{Remark}
\newtheorem{algorithm}[theorem]{Algorithm}

\numberwithin{equation}{section}
\title[Generalized relaxation of string averaging operators]{Generalized relaxation of string averaging operators based on strictly relaxed cutter operators}
\author[T. Nikazad]{Touraj Nikazad}

\address[T. Nikazad]{School of Mathematics, Iran University of Science and
Technology, 16846-13114, Tehran, Iran.} \email{{\tt
tnikazad@iust.ac.ir}}

\author[M. Mirzapour]{Mahdi Mirzapour}
\address[M. Mirzapour]{School of Mathematics, Iran University of Science and
Technology, 16846-13114, Tehran, Iran.} \email{\tt
mahdimirzapour67@gmail.com}

\keywords{convex feasibility problem; relaxed cutter
operator; strictly quasi-nonexpansive operator; string averaging;
generalized relaxation}

\subjclass[2010]{90C25; 47J25; 49M20}

\begin{document}
\begin{abstract}
We present convergence analysis of a generalized relaxation of
string averaging operators which is based on strictly relaxed
cutter operators on a general Hilbert space. In this paper, the
string averaging operator is assembled by averaging of strings'
endpoints and each string consists of composition of finitely many
strictly relaxed cutter operators. We also consider projected
version of the generalized relaxation of string averaging
operator. To evaluate the study, we recall a wide class of
iterative methods for solving linear equations (inequalities) and
use the subgradient projection method for solving nonlinear convex
feasibility problems.
\end{abstract}

\maketitle

\section{Introduction}
In this paper we consider a fixed point iteration method for
solving convex feasibility problems which are used in different
areas of mathematics and physical sciences. A convex feasibility
problem consists in finding a point in the intersection of closed
convex sets $\{C_{\ell}\}_{\ell=1}^{N}.$ Using string averaging
method \cite{string2001}, which is particularly suitable for
parallel computing and therefore have the ability to handle
huge-size problems, may accelerate the fixed point iteration
method. The output of the string averaging process is an operator,
called string averaging operator, which is used in the fixed point
iteration method. In this paper, the string averaging operator is
made by averaging of finitely many operators which are composition
of finitely many strictly relaxed cutter operators. The string
averaging process is studied in many research works as
\cite{BMR,BJ,sumper2007,string2001,inco-string,CZ2014,dynsumper2013}
which are based on projection operators. Recently in
\cite{Crombez2002}, \cite{NM2014} and \cite{ys2008}, a fixed point
iteration method is analyzed based on strict paracontraction
operators, strictly quasi-nonexpansive operators and cutter
operators respectively. Moreover, the string averaging scheme has
been extended in \cite{inco-string}, \cite{NA2015} and
\cite{RZ2015}. The cutter operators are introduced and
investigated in \cite{BC} and studied in several research works as
\cite{BCk,CC} and references therein.

To accelerate a fixed point iteration algorithm, one may use
relaxation parameters or its generalized version which is called
generalized relaxation. The generalized relaxation strategy is
recently studied for composition of cutter operators in \cite{CC}
and in \cite{NM2015Conf} without considering string averaging
process. On the other hand, various extrapolation schemes applied
to a pure convex combination can be found in literature; see, for
example,
\cite{Aleyner2008,HB1996,CC2011,Combettes1996,Combettes1997hilbertian,Com1997Extra,EKN,Pierra1984}
and \cite[Chapter 4, Chapter 5]{C2013}. A recent work
\cite{NM2014} analyzes a fixed point iteration method based on
generalized relaxation of the string averaging operator which is
based on strictly quasi-nonexpansive operators. Note that the set
of strictly quasi-nonexpansive operators involves all cutter
operators.  The analysis in \cite{NM2014} indicates that the
generalized relaxation of cutter operators is inherently able to
make more acceleration comparing with \cite{CC}, see section
\ref{disc} for more details.

We study a fixed point iteration method based on generalized
relaxation of string averaging operator using strictly relaxed
cutter operators. The class of relaxed cutter operators contains
relaxed projection operators \cite{Cimmino}, relaxed subgradient
projections \cite{cl81,NA2013}, relaxed firmly nonexpansive
operators \cite{BR1977}, the resolvents of a maximal monotone
operators \cite{BCk,EB1992}, contraction operators
\cite{BB,C2013}, averaged operators \cite{BBR1987} and strongly
quasi-nonexpansive operators \cite{BB}.

One may ask: Is there any connection between the set of strictly
relaxed cutter operators and the set of strictly
quasi-nonexpansive operators? The answer is yes. Actually any
strictly relaxed cutter operator is strictly quasi-nonexpansive
operator, see Remark \ref{strongtostrct}. On the other hand, the
convergence analysis of a fixed point iteration method based on
the generalized relaxation of the string averaging operator has
been shown in \cite{NM2014} where each string consists of
composition of finitely many strictly quasi-nonexpansive
operators. Therefore, the next question is: What are the
advantages of using strictly relaxed cutter operators instead of
strictly quasi-nonexpansive operators? In \cite{NM2014}, the
relaxation parameters are chosen from $(0,1)$ whereas our
relaxation parameters lie in $(0,2).$  Also, we will show that
using strictly relaxed cutter operators gives faster reduction in
error. The next advantage of these operators is related to
demi-closedness property. In the string averaging process we have
finitely many operators, say $\{U_t\}_{t=1}^E,$ which are
composition of some finitely many other operators, i.e.,
$U_t=T_{m_t}\cdots T_{1}.$ Also, there is an averaging process on
$U_t$ which makes the string averaging operator, say $T.$ We
consider a fixed point iteration method based on the generalized
relaxation of $T.$ To get convergence result for the fixed point
iteration method, it is assumed in \cite{NM2014} that $T-Id$ is
demi-closed at zero or alternatively $\{U_t-Id\}_{t=1}^E$ are
demi-closed at zero . In addition to demi-closedness of $T-Id$ or
$\{U_t-Id\}_{t=1}^E,$ using strictly relaxed cutter operators
allows us to consider, similar to \cite{CC}, the demi-closedness
of all operators $T_i-Id.$

The paper is organized as follows. In section \ref{pre} we recall
some definitions and properties of relaxed cutter operators. We
reintroduce string averaging process and give its convergence
analysis based on strictly relaxed cutter operators in section
\ref{main}. A short discussion on error reduction and choosing
relaxation parameter are presented in section \ref{disc}. In
section \ref{const} we present the projected version of
generalized relaxation of string averaging operator with
convergence proof. At the end, the capability of the main result
is examined in section \ref{apl} by employing the subgradient
projection method.

\section{Preliminaries and Notations}\label{pre}
Throughout this section, we consider $T:H \rightarrow H$ with
nonempty fixed point set, i.e., $Fix T\not = \emptyset$ where $H$
is a Hilbert space and $Id$ denotes the identity operator on $H$.
The following definitions, see \cite{C2013}, will be useful in our
future analysis.
\begin{definition}\label{tarifstrictly}
An operator $T$ is quasi-nonexpansive (QNE) if
\begin{equation}\label{t1}
\|T(x)-z\|\leq\|x-z\|
\end{equation}
for all $x\in H$ and $z\in Fix T.$ Also, one may use the term
strictly quasi-nonexpansive (sQNE) by replacing strict inequality
in (\ref{tarifstrictly}), i.e., $\|T(x)-z\|<\|x-z\|$ for all $x\in
H\backslash Fix T$ and $z\in Fix T.$ Moreover, a continuous sQNE
operator is called paracontracting operator, see \cite{EKN}.
\end{definition}
Another useful class of operators is the class of cutter
operators, namely,
an operator $T:H \rightarrow H$ with nonempty fixed point set is
called  $cutter$ if
\begin{equation}\label{cutterdefinition}
\left<x-T(x), z-T(x)\right> \leq 0
\end{equation}
for all $x \in H$ and $z\in FixT$. Using \cite[Remark
2.1.31]{C2013}, the operator $T$ is a cutter if and only if
\begin{equation}\label{cutterdefinition2}
\left<T(x)-x, z-x\right> \geq \|T(x)-x\|^2
\end{equation}
for all $x \in H$ and $z\in Fix T.$
\begin{definition}\label{alp-relax}
Let $T: H\rightarrow H$ and $\alpha\in [0,2]$. The following
operator
\begin{equation}
T_{\alpha}:=(1-\alpha)Id+\alpha T
\end{equation}
is called an $\alpha$-relaxation or, shortly, relaxation of the
operator $T$. If $\alpha\in (0,2)$, then $T_{\alpha}$ is called a
strictly (or strict) relaxation of $T$.
\end{definition}
Based on \cite[Remark 2.1.31]{C2013}, an $\alpha$-relaxed cutter
operator is defined as follows.
\begin{definition}\label{alfa-cutter}
Let $T:H \rightarrow H$ has a fixed point. Then the operator $T$
is an $\alpha$-relaxed  cutter, or, shortly, relaxed cutter where
$\alpha \in [0,2]$, if
\begin{equation}\label{alpharelaxedcutter}
\left<T_{\alpha}(x)-x, z-x\right>=\alpha\left<T(x)-x, z-x\right>
\geq \|T(x)-x\|^2
\end{equation}
for all $x \in H$ and  $z\in Fix T$. If $\alpha\in (0,2)$, then
$T_{\alpha}$ is called a strictly relaxed cutter operator of $T$.
\end{definition}
Let $\alpha\geq0$ and assume that $T:H \rightarrow H$ has a fixed
point. We say that $T$ is $\alpha$-strongly quasi-nonexpansive
($\alpha-$SQNE), if
\begin{equation}\label{SQNE1}
\|T(x)-z\|^{2}\leq\|x-z\|^{2}-\alpha\|T(x)-x\|^{2}
\end{equation}
for all $x\in H$ and $z \in Fix T.$ Also, the operator $T$
satisfying (\ref{SQNE1}) with $\alpha>0$ is called strongly
quasi-nonexpansive (SQNE) operator.

Following theorem presents a relationship between strictly relaxed
cutter and SQNE operators.
\begin{theorem}\cite[Theorem 2.1.39 and Corollary 2.1.40]{C2013}\label{cuttertosqne}
Assume that $T:H \rightarrow H$ has a fixed point and let $\lambda
\in (0,2]$. Then $T$ is a ${\lambda}$-relaxed cutter if and only
if $T$ is $\frac{2-\lambda}{\lambda}$-SQNE, i.e.,
\begin{equation}
\|T_{\lambda}(x)-z\|^{2}\leq\|x-z\|^{2}-\frac{2-\lambda}{\lambda}\|T_{\lambda}(x)-x\|^{2}
\end{equation}
for all $x\in H$ and all $z \in Fix T$.
\end{theorem}
The following remark gives a relationship between sQNE and
strictly relaxed cutter operators.
\begin{remark}\cite[Remark 2.1.44.]{C2013}\label{strongtostrct}
Assume that $T:H \rightarrow H$ has a fixed point. If $T$ is SQNE,
then $T$ is sQNE. Therefore, all properties of sQNE operators are
also valid for SQNE and strictly relaxed cutter operators.
\end{remark}

A very useful property of $\alpha$-relaxed cutter and sQNE
operators is their closedness respect to convex combination and
composition of the operators. Furthermore, any cutter operator is
$1$-relaxed cutter, compare \cite[Remark 2.1.31]{C2013} and
Definition \ref{alfa-cutter}. However, the class of cutter
operators is not necessarily closed with respect to composition of
operators. The closedness property of sQNE operators is presented
in \cite[Theorem 2.1.26]{C2013}. Also, the class of strictly
relaxed cutter operators has the closedness property, see
following theorem.
\begin{theorem}\label{relaxedcutter}\cite[Theorem 2.1.48 and Theorem 2.1.50]{C2013}
Let $L_{\alpha_{i}}:X\rightarrow X$ be an $\alpha_{i}$-relaxed
cutter, where $\alpha_{i}\in (0,2)$ and $i\in I=\{1,2, \ldots,
m\}$. Let $\bigcap_{i\in I}Fix L_{\alpha_{i}}\neq \emptyset$ and
$P_{m}:=L_{\alpha_{m}}L_{\alpha_{m-1}}\ldots L_{\alpha_{1}}$. Then
the operator $P_{m}$ is a $\gamma_{m}$-relaxed cutter, with
\begin{equation}\label{gamma}
\gamma_{m}=\frac{2}{(\sum_{i=1}^m\frac{\alpha_{i}}{2-\alpha_{i}})^{-1}+1}.
\end{equation}
Furthermore, the operator $Q_{m}=\sum_{i\in I}
\omega_{i}L_{\alpha_{i}}$ is a $\mu$-relaxed cutter where
$\mu=\sum_{i\in I} \omega_{i}\alpha_{i}$, $ \sum_{i\in
I}\omega_i=1$ and $\omega_i\geq 0$. Moreover,
$FixP_{m}=FixQ_{m}=\bigcap_{i\in I}FixL_{\alpha_{i}}$.
\end{theorem}

We next reintroduce, see also \cite{C2013,C2010,CC2011,CC,NM2014},
the generalized relaxation of an operator, which allows to
accelerate locally a fixed point iteration method.
\begin{definition}\label{def1}
Let $T: H \rightarrow H$ and  $\sigma:H\rightarrow (0,\infty)$ be
a {\it step size function.} The generalized relaxation of $T$ is
defined by
\begin{equation}\label{t4}
T_{\sigma,\lambda}(x)=x+\lambda \sigma(x)(T(x)-x)
\end{equation}
where  $\lambda$ is a relaxation parameter in $[0,2].$
\end{definition}
If $\lambda\sigma(x) \geq 1$ for all $x\in H$, then the operator
$T_{\sigma,\lambda}$ is called an {\it extrapolation} of $T$. For
$\sigma(x)=1$ we get the relaxed version of $T,$ namely,
$T_{1,\lambda}=:T_{\lambda}.$ Furthermore, it is clear that
$T_{\sigma,\lambda}(x)=x+\lambda(T_{\sigma}(x)-x)$ where
$T_{\sigma}=T_{\sigma,1}$ and $Fix~T_{\sigma,\lambda}=Fix T$ for
any $\lambda\neq0.$

\begin{definition}\label{def2}
An operator $T: H \rightarrow H$ is {\it demi-closed} at $0$ if
for any weakly converging sequence $x^k\rightharpoonup y\in H$
with $T(x^k)\rightarrow 0$  we have $T(y)=0$.
\end{definition}
\begin{remark}\cite[p. 108]{C2013}\label{rem2}
It is well known that the operator $T - Id$ is demi-closed at $0$
where $T: H \rightarrow H$ is a nonexpansive operator.
\end{remark}

\begin{remark}\label{demi_closeness}
Assume that $\{T_{i}\}_{i=1}^{m}$ are strictly relaxed cutter
operators such that $\{T_{i}-Id\}_{i=1}^{m}$ are demi-closed at
$0$ and $\bigcap_{i=1}^{m} FixT_{i}\neq \emptyset$. Based on
\cite[Theorem 4.1, Theorem 4.2]{Ceg2015}, \cite[Lemma 3.4]{RZ2015}
and Theorem \ref{cuttertosqne}, $V-Id$ is demi-closed at $0$ while
assuming that $V$ is defined either by a composition
$V=T_{m}\ldots T_{1}$ or by a convex combination
$V=\sum_{i=1}^{m}\omega_{i}T_{i}$.
\end{remark}

\section{Main result}\label{main}
In this section we reintroduce the string averaging precess which
is based on the class of strictly relaxed cutter operators. We
next consider a fixed point iteration method based on the
generalized relaxation of string averaging operator and present
its convergence analysis. We give the projected version of
generalized relaxation of string averaging operator with
convergence analysis. We also compare the error reduction of our
algorithm with \cite{CC} and \cite{NM2014} and give a short
discussion on choosing relaxation parameters of strictly relaxed
cutter operators.

We first give a short review of research works on the string
averaging algorithm. The string averaging algorithmic scheme is
first proposed in \cite{string2001}. Their analysis was based on
the projection operators, whereas the algorithm is defined for any
operators, for solving consistent convex feasibility problems.
Studying the algorithm in a more general setting is considered by
\cite{BMR}.  The inconsistent case is analyzed by
\cite{inco-string} and they proposed a general algorithmic scheme
for string averaging method without any convergence analysis. A
special case of the algorithm is studied under summable
perturbation in \cite{sumper2007,DHC2009}. A dynamic version of
the algorithm is presented in \cite{CA2014}. In \cite{Gordon2005}
the string averaging method is compared with other methods for
sparse linear systems. Other applications of the string averaging
scheme, such as constrained minimization and variational
inequalities can be found, for example, in \cite{CZ2014} and
\cite{CegZ2014}, respectively.

Recently, a perturbation resilience iterative method with an
infinite pool of operators is studied in \cite{NA2015} which
answers some open problems mentioned by \cite{string2001} whereas
these problems are partially answered by \cite{dynsumper2013}.
Also the proposed general algorithmic scheme of \cite[Algorithm
3.3]{inco-string}, which was presented without any convergence
analysis, is extended with a convergence proof in \cite{NA2015}.
Another general form of the string averaging scheme appeared in
\cite{RZ2015}.


All the above mentioned research works are based on projection
operators. In \cite{ys2008}, the string averaging algorithm is
studied for cutter operators and the sparseness of the operators
is used in averaging process. In \cite{Crombez2002,Crombez2004},
the string averaging method is used for finding common fixed point
problem of strict paracontraction operators.

We next reintroduce the string averaging algorithm as follows.

\begin{definition}\label{tarifstring}
The string $I^{t}=(i_1^t,i^t_2,...,i_{m_t}^t)$ is an ordered
subset of $I=\{1,2,...,m\}$ such that $\bigcup^{E}_{t=1}I^t=I$.
Define
\begin{eqnarray}\label{tarifstring2}
U_{t}&=&T_{i^{t}_{m_{t}}}...T_{i^{t}_{2}}T_{i^{t}_{1}},~t=1, 2,
\ldots, E\nonumber\\
T&=&\displaystyle\sum_{t=1}^{E} \omega_{t}U_{t}
\end{eqnarray}
where $\omega_{t}>0$ and $\sum_{t=1}^{E} \omega_{t}=1$. Here
$T_{i\in I}$ are operators on a Hilbert space $H.$
\end{definition}
In this paper, we assume that all $T_{i\in I}$ of Definition
\ref{tarifstring} are strictly relaxed cutter operators on $H$ and
$\bigcap_{i\in I}FixT_{i}\neq \emptyset$. It should be mentioned
that the averaging process (\ref{tarifstring2}) is a special case
of \cite{string2001}.
\begin{remark}\label{rem3}
Note that all $\{U_t\}_{t=1}^E$ and consequently the operator $T$
belong to the class of strictly relaxed cutter operators where
$T_{i\in I}$ are strictly relaxed cutter operators, see Theorem
\ref{relaxedcutter}.
\end{remark}
We now consider the following fixed point iteration algorithm
which is based on generalized relaxation of $T.$
\begin{algorithm}\label{algt1}
 \emph{\\Initialization:} $x^{0}\in H$ is
arbitrary. \newline \emph{Iterative Step:} Given $x^{k},$ compute
\begin{eqnarray}
x^{k+1}=T_{\sigma,\lambda_k}(x^k).\nonumber
\end{eqnarray}
\end{algorithm}
To simplifying the notation of Definition \ref{tarifstring}, we
denote $i^t_{\ell}$ by $\ell$ for $\ell=1,2,\cdots,m_t.$ Analogues
with \cite{CC}, we consider the following notations
\begin{equation}\label{notation1}
S_{0}=Id,~S_{i}=T_{i}\ldots T_{1} \textrm{ for }i=1, \ldots, m_{t}
\end{equation} which leads to $U_{t}=S_{m_{t}}$.
Furthermore, let
\begin{equation}\label{notation}
u^{0}=x,~u^{i}=T_{i}u^{i-1},~y^{i}=u^{i}-u^{i-1}
\end{equation}
for $ i=1,\ldots, m_{t}$ and $ t=1, \ldots, E.$ Using
(\ref{notation1}) and (\ref{notation}) we have
\begin{equation}\label{notation2}
u^{i}=S_{i}x,~ \sum_{j=i}^{m_{t}}y^{j}=U_{t}(x)-S_{i-1}x \textrm{
for } i=1,2, \ldots, m_{t}
\end{equation}
and particularly $\sum_{j=1}^{m_{t}}y^{j}=U_{t}(x)-x.$
\begin{lemma}\label{lemma:gam}
Let $T_{i\in I}:H\rightarrow H$ be $\alpha_i$-relaxed cutter
operators such that $\bigcap_{i=1}^{m} Fix T_{i}\neq \emptyset$
and $\alpha_i \in (0,2).$ For any $z\in \bigcap_{i=1}^{m} Fix
T_{i}$ we have
\begin{eqnarray}
\left<z-x,T(x)-x\right>&\geq&\sum_{t=1}^{E}\omega_{t}\sum_{i=1}^{m_{t}} \left<\frac{1}{\alpha_{i}}y^{i}+\sum_{j=i+1}^{m_{t}} y^{j},y^{i}\right>\label{lemmaasli0}\\
&\geq&\sum_{t=1}^{E}\omega_{t} \sum_{i=1}^{m_{t}}\left(\frac{1}{\alpha_{i}}-\frac{1}{2}\right)\|y^{i}\|^{2}\label{lemmaasli00}\\
&\geq&\sum_{t=1}^{E}\omega_{t}\left(\frac{1}{\bar{\alpha}_{t}}-\frac{1}{2}\right) \sum_{i=1}^{m_{t}}\|y^{i}\|^{2}\nonumber\\
&\geq& \left(\frac{1}{\gamma}-\frac{1}{2}\right)\sum_{t=1}^{E}\frac{\omega_{t}}{m_{t}} \|U_{t}(x)-x\|^2\nonumber\\
&\geq&\displaystyle\frac{\left(\frac{1}{\gamma}-\frac{1}{2}\right)}{\bar{m}_{t}}\|T(x)-x\|^2\label{lemmaasli}
\end{eqnarray}
where $\bar{\alpha_{t}}:=max_{1\leq i\leq m_t}\alpha_{i}$,
$\gamma:=max_{1\leq t\leq E}\bar{\alpha}_{t}$ and
$\bar{m}_{t}:=max_{1\leq t\leq E}m_{t}.$
\end{lemma}
\begin{proof}
We first assume $E=1$ which leads to the case $m_t=m.$ The
inequality (\ref{lemmaasli}) is directly followed by Definition
\ref{alfa-cutter} where $m=1$. Now let $m=2.$ Since $T_i$ is an
$\alpha_{i}$-relaxed cutter operator and using (\ref{notation}),
we have
\begin{eqnarray}
\left<z-x, T_{2}T_{1}(x)-x\right>&=&\left<z-x, T_{2}T_{1}(x)-T_{1}(x)+T_{1}(x)-x\right>\nonumber\\
&=&\left<z-T_{1}(x), T_{2}T_{1}(x)-T_{1}(x)\right>+\left<z-x, T_{1}(x)-x\right>\nonumber\\
&+& \left<T_{1}(x)-x, T_{2}T_{1}(x)-T_{1}(x)\right>\nonumber\\
&\geq& \frac{1}{\alpha_{2}}\|T_{2}T_{1}(x)-T_{1}(x)\|^2+\frac{1}{\alpha_{1}}\|T_{1}(x)-x\|^2\nonumber\\
&+&\left<T_{1}(x)-x, T_{2}T_{1}(x)-T_{1}(x)\right>\nonumber\\
&=& \frac{1}{\alpha_{2}}\|y^{2}\|^2+\frac{1}{\alpha_{1}}\|y^{1}\|^2+\left<y^{1}, y^{2}\right>\nonumber\\
&=&\sum_{i=1}^{2}\left<\frac{1}{\alpha_{i}}y^{i}+\sum_{j=i+1}^2 y^j, y^{i}\right>\nonumber\\
&=&\sum_{i=1}^{2}\left<\frac{1}{\alpha_{i}}y^{i}+\sum_{j=i+1}^2 y^j, y^{i}\right>\nonumber\\
&-&\frac{1}{2}\left(\|y^{1}\|^2+\|y^{2}\|^2\right)+\frac{1}{2}\left(\|y^{1}\|^2+\|y^{2}\|^2\right)\nonumber\\
&=&\left(\frac{1}{\alpha_{2}}-\frac{1}{2}\right)\|y^{2}\|^2+\left(\frac{1}{\alpha_{1}}-\frac{1}{2}\right)\|y^{1}\|^2\nonumber\\
&+&\frac{1}{2}\|y^{1}+y^{2}\|^2\nonumber\\
&\geq& \left(\frac{1}{\alpha_{2}}-\frac{1}{2}\right)\|y^{2}\|^2+ \left(\frac{1}{\alpha_{1}}-\frac{1}{2}\right)\|y^{1}\|^2\nonumber\\
&\geq& \left(\frac{1}{\bar{\alpha}}-\frac{1}{2}\right)\left(\|y^{1}\|^2+\|y^{2}\|^2\right)\nonumber\\
(\textrm{using convexity of }\|.\|^2)&\geq&
\frac{1}{2}\left(\frac{1}{\bar{\alpha}}-\frac{1}{2}\right)\|y^{1}+y^{2}\|^2\label{m=2}\nonumber
\end{eqnarray}
where $\sum_{j=\ell}^k y^j=0$ for $k<\ell$ and
$\bar{\alpha}=\max_{1\leq i \leq 2} \alpha_{i}$. Now suppose that
the inequality (\ref{lemmaasli0}) holds for $m=t$ (induction
hypothesis). Similar to \cite[Lemma 7]{CC}, define $V_1=Id$,
$V_{i}=T_{i}T_{i-1} \cdots T_{2}$ for $i = 2, 3, \ldots, t + 1$
and put $v^{1}= h$, where $h$ is an arbitrary element of H, $v^{i}
= T_{i} v^{i-1}$ and $z^{i} = v^{i} - v^{i-1}$, $i = 2, 3,\ldots,
t + 1$. If we set $h = T_{1}x$, then $S_{i}x = V_{i}h$, $u^{i} =
v^{i}$ and $y^{i} = z^{i}$ for $i = 2, 3, \ldots, t + 1$. Using
the induction hypothesis we have
\begin{equation}\label{dec1}
\left<V_{t+1}h-h, z-h\right> \geq
\sum_{i=2}^{t+1}\left<\frac{1}{\alpha_{i}}z^{i}+
\sum_{j=i+1}^{t+1} z^j, z^{i}\right>
\end{equation}
for all $h \in H$ and $z \in\bigcap_{i=2}^{t+1}Fix T_{i}$. If
$h=T_{1}x$ then for $x \in H$ and $z\in \bigcap_{i=1}^{t+1}Fix
T_{i}$ we obtain that
\begin{eqnarray}
\left<S_{t+1}x-x, z-x\right>&=&\left< V_{t+1}h-x,z-x\right>\nonumber\\
&=&\left< V_{t+1}h-h,z-x\right>+\left<T_{1}x-x,z-x\right>\nonumber\\
&\geq& \left< V_{t+1}h-h,z-h\right>
+\left< V_{t+1}h-h,h-x\right> + \frac{1}{\alpha_{1}}\|y^{1}\|^2\nonumber\\
(\textrm{using }(\ref{dec1}))&\geq& \sum_{i=2}^{t+1}\left<\frac{1}{\alpha_{i}}z^{i}+ \sum_{j=i+1}^{t+1} z^j, z^{i}\right>
+\left<\sum_{i=2}^{t+1}z^{i}, h-x\right>+ \frac{1}{\alpha_{1}}\|y^{1}\|^2\nonumber\\
&=&\sum_{i=2}^{t+1}\left<\frac{1}{\alpha_{i}}y^{i}+ \sum_{j=i+1}^{t+1} y^j, y^{i}\right>
+\left<\sum_{i=2}^{t+1}y^{i}, y^{1}\right>+ \frac{1}{\alpha_{1}}\|y^{1}\|^2\nonumber\\
&=& \sum_{i=1}^{t+1}\left<\frac{1}{\alpha_{i}}y^{i}+
\sum_{j=i+1}^{t+1} y^j, y^{i}\right>.\label{induction}
\end{eqnarray}
Therefore, the inequality (\ref{lemmaasli0}) is true for $m=t+1$
which completes the induction. The inequality (\ref{lemmaasli00})
can be obtained by
\begin{eqnarray}
\sum_{i=1}^{m}\left<\frac{1}{\alpha_{i}}y^{i}+ \sum_{j=i+1}^{m}
y^j,y^{i}\right>
&=&\sum_{i=1}^{m}\left<\frac{1}{\alpha_{i}}y^{i}+ \sum_{j=i+1}^{m} y^j, y^{i}\right>\label{dec3}\\
&-&\frac{1}{2}\sum_{i=1}^{m}\|y^{i}\|^2+\frac{1}{2}\sum_{i=1}^{m}\|y^{i}\|^2\nonumber\\
&=&\sum_{i=1}^{m}\left(\frac{1}{\alpha_{i}}-\frac{1}{2}\right)\|y^{i}\|^2+\frac{1}{2}\|\sum_{i=1}^{m}y^{i}\|^2\nonumber\\
&\geq& \left(\frac{1}{\bar{\alpha}}-\frac{1}{2}\right)\sum_{i=1}^{m}\|y^{i}\|^2\nonumber\\
(\textrm{using convexity of }\|.\|^2)&\geq&
\frac{1}{m}\left(\frac{1}{\bar{\alpha}}-\frac{1}{2}\right)\|\sum_{i=1}^{m}y^{i}\|^2.\label{dec4}
\end{eqnarray}
Now we complete the proof by considering the case $E\geq 1.$ Using
Definition \ref{tarifstring}, we have
\begin{eqnarray}
\left<z-x, T(x) -x\right>&=&\sum_{t=1}^{E}\omega_{t}\left<z-x, U_{t}(x) -x\right>\label{lemm1}\\
(\textrm{using }(\ref{induction}))&\geq&\sum_{t=1}^{E}\omega_{t}\sum_{i=1}^{m_{t}}\left<\frac{1}{\alpha_{i}}y^{i}+\sum_{j=i+1}^{m_{t}} y^j, y^{i}\right>\label{lemm2}\\
(\textrm{using }(\ref{dec3}),(\ref{dec4}))&\geq&\sum_{t=1}^{E}\omega_{t}\left(\frac{1}{\bar{\alpha_{t}}}-\frac{1}{2}\right)\sum_{i=1}^{m_{t}}\|y^{i}\|^2\label{dec2}\\
&\geq&\sum_{t=1}^{E}\frac{\omega_{t}}{m_{t}}\left(\frac{1}{\bar{\alpha_{t}}}-\frac{1}{2}\right)\|\sum_{i=1}^{m_{t}}y^{i}\|^2\nonumber\\
&=&\sum_{t=1}^{E}\frac{\omega_{t}}{m_{t}}\left(\frac{1}{\bar{\alpha_{t}}}-\frac{1}{2}\right)\|U_{t}(x)-x\|^2\nonumber\\
&\geq&\frac{1}{\bar{m}_{t}}\left(\frac{1}{\gamma}-\frac{1}{2}\right)\sum_{t=1}^{E}\omega_{t}\|U_{t}(x)-x\|^2\label{dec6}\\
(\textrm{using convexity of }\|.\|^2)&\geq&\frac{1}{\bar{m}_{t}}\left(\frac{1}{\gamma}-\frac{1}{2}\right)\|\sum_{t=1}^{E}\omega_{t}\left(U_{t}(x)-x\right)\|^2\nonumber\\
&=&\frac{1}{\bar{m}_{t}}\left(\frac{1}{\gamma}-\frac{1}{2}\right)\|T(x)-x\|^2\label{dec5}
\end{eqnarray}
which completes the proof.
\end{proof}

We next show that the generalized relaxation operator $T_{\sigma,
\lambda}$ is a $\lambda$-relaxed cutter operator under a condition
on $\sigma(x)$.

\begin{theorem}\label{lemma3}
Let $T_{\sigma,\lambda}$ be a generalized relaxation of
$T=\sum_{t=1}^{E} \omega_{t}U_{t}$ and $\bigcap_{i\in I}Fix
T_{i}\neq \emptyset$. Then $T_{\sigma,\lambda}$ is a
$\lambda$-relaxed cutter operator if
\begin{equation}\label{strictsigma}
0<\sigma(x)<\displaystyle\frac{\sum^{E}_{t=1}w_{t}\sum^{m_{t}}_{i=1}
\left<U_{t}(x)-S_{i}(x)+\frac{1}{\alpha_{i}}\zeta_i(x),
\zeta_i(x)\right>}{\|T(x)-x\|^{2}}
\end{equation}
 where $x\in H\backslash Fix
T$ and $\zeta_i(x)=S_{i}(x)-S_{i-1}(x)$. Furthermore, the step
size function
\begin{equation}\label{sigmamax}
\sigma_{max}(x) = \left\{ \begin{array}{ll}
\displaystyle\frac{\sum^{E}_{t=1}w_{t}\sum^{m_{t}}_{i=1}
\left<U_{t}(x)-S_{i}(x)+\frac{1}{\alpha_{i}}\zeta_i(x),
\zeta_i(x)\right>}
{\|T(x)-x\|^{2}}, & x\in H\backslash Fix T \vspace{0.1cm}\\
1, & x\in Fix T
\end{array} \right.
\end{equation}
is bounded below by
$\frac{1}{\bar{m}_{t}}\left(\frac{1}{\gamma}-\frac{1}{2}\right).$
\end{theorem}

\begin{proof}
For $z\in Fix~T_{\sigma,\lambda}$ and $x\in H\backslash
Fix~T_{\sigma,\lambda}$ one gets
\begin{eqnarray}
\left<z-x,T_{\sigma,\lambda}(x)-x\right>&=&\left<z-x,\lambda\sigma(x)(T(x)-x)\right>\nonumber\\
&=&\lambda\sigma(x)\left<z-x,T(x)-x\right>\nonumber\\
(\textrm{using }(\ref{lemmaasli0}))&\geq& \lambda\sigma(x)\sum_{t=1}^{E}\omega_{t}\sum_{i=1}^{m_{t}}\left<\frac{1}{\alpha_{i}}y^{i}+\sum_{j=i+1}^{m_{t}} y^j, y^{i}\right>\label{s1}\\
&=&\lambda\sigma(x)\displaystyle\sum^{E}_{t=1}w_{t}\sum^{m_{t}}_{i=1}
\left<U_{t}(x)-S_{i}(x)+\frac{1}{\alpha_{i}}\zeta_i(x),\zeta_i(x)\right>\label{s2}\\
(\textrm{using }(\ref{strictsigma}))&\geq&\frac{\lambda^{2}\sigma^{2}(x)}{\lambda}\|T(x)-x\|^2\\
&=&\frac{1}{\lambda}\|T_{\sigma,\lambda}(x)-x\|^2.
\end{eqnarray}
So $T_{\sigma,\lambda}$ is a $\lambda$-relaxed cutter operator.
Using (\ref{lemm2}), (\ref{dec5}), (\ref{s1}) and (\ref{s2}), the
lower bound of $\sigma_{max}$ is
\begin{eqnarray}
\sigma_{max}(x) &\geq& \frac{\sum_{t=1}^{E}\omega_{t}\left(\frac{1}{\bar{\alpha}_{t}}-\frac{1}{2}\right) \sum_{i=1}^{m_{t}}\|y^{i}\|^2}{\|T(x)-x\|^2}\\
&\geq&\frac{\frac{1}{\bar{m}_{t}}\left(\frac{1}{\gamma}-\frac{1}{2}\right)\|T(x)-x\|^2}{\|T(x)-x\|^2}\\
&\geq&
\frac{1}{\bar{m}_{t}}\left(\frac{1}{\gamma}-\frac{1}{2}\right)
 \end{eqnarray}
where $x\in H\backslash Fix T.$ It completes the proof.
\end{proof}

\begin{corollary}\label{col}
$T_{\sigma,\lambda} - Id$ is demi-closed at $0$ assuming that
$T-Id$ is demi-closed at $0$. This happens, for example, when
$T_{i}-Id$ is demi-closed at $0$ for $i=1,\ldots,m$, or $U_{t}-Id$
is demi-closed at $0$ for $t=1,\ldots,E$.
\end{corollary}
The proof of Corollary \ref{col} follows immediately by Remark
\ref{demi_closeness} and the estimate established in Theorem
\ref{lemma3} , that is, we have
\begin{equation*}
\|T_{\sigma,\lambda}(x)-x\|=\sigma(x)\lambda\|T(x)-x\|\geq
\lambda\frac{1}{\bar{m}_{t}}\left(\frac{1}{\gamma}-\frac{1}{2}\right)\|T(x)-x\|.
\end{equation*}
\begin{theorem}\label{theorem2}
Let $\sigma=\sigma_{max}$ be step size function, $\bigcap_{i\in I}
FixT_{i}\neq \emptyset$ and
$\lambda_{k}\in[\varepsilon,2-\varepsilon]$ for an arbitrary
constant $\varepsilon\in(0,1)$. The sequence generated by
Algorithm \ref{algt1}  converges weakly to a point in $Fix T$, if
one of the following conditions is satisfied
\begin{itemize}
\item[($i$)] $T-Id$ is demi-closed at $0$, or
\item[($ii$)] $U_{t} -Id$ are demi-closed at $0$, for all $t=1, 2, \cdots, E$, or
\item[($iii$)] $T_{i} -Id$ are demi-closed at $0$, for all $i=1, 2, \cdots, m$.
\end{itemize}
\end{theorem}
\begin{proof}
For $z\in Fix T, x^k\in H\backslash Fix T$ one gets
\begin{eqnarray}\label{aslih2}
e_{k+1}^2 &=& \|T_{\sigma,\lambda_k}(x^k)-z\|^2\nonumber\\
&=&\|x^{k}+\lambda_{k}\sigma(x^k)(T(x^k)-x^{k})-z\|^2\nonumber\\
&=& \|x^{k}+\lambda_{k}(T_{\sigma}(x^k)-x^{k})-z\|^2\nonumber\\
&=&e_{k}^2 +\xi_k +\lambda^{2}_{k}\|T_{\sigma}(x^k)-x^{k}\|^2
\end{eqnarray}
where $\xi_k=2 \lambda_{k}\left<x^k
-z,T_{\sigma}(x^k)-x^{k}\right>$ and $e_k=\|x^k-z\|.$ Using
Theorem \ref{lemma3}, we obtain that
$$\left<x^k
-z,T_{\sigma}(x^k)-x^{k}\right>\leq -
\|T_{\sigma}(x^{k})-x^{k}\|^2$$ and consequently
\begin{equation}\label{aslih22}
\xi_{k}\leq -2\lambda_{k}\|T_{\sigma}(x^k)-x^{k}\|^2.
\end{equation}
Using (\ref{aslih22}) and Theorem \ref{lemma3} we conclude
\begin{eqnarray}
e_{k+1}^2 &\leq&\|x^{k}-z\|^2 -2\lambda_{k}\|T_{\sigma}(x^k)-x^{k}\|^2+\lambda^{2}_{k}\|T_{\sigma}(x^k)-x^{k}\|^2\nonumber\\
&=&e_{k}^2-\lambda_{k}(2-\lambda_{k})\sigma^{2}(x^k)\|T(x^k)-x^{k}\|^2 \label{t10}\\
&\leq& e_{k}^2-
\frac{\lambda_{k}(2-\lambda_{k})}{\bar{m_{t}}^{2}}\left(\frac{1}{\gamma}-\frac{1}{2}\right)^{2}\|T(x^k)-x^{k}\|^2.
\label{demiT}
\end{eqnarray}
Therefore, we get that $\left\{e_k\right\}$ decreases and
consequently $\left\{x^{k}\right\}$ is bounded and
\begin{equation}\label{demiclosedT}
\|T(x^k)-x^k\|\rightarrow 0.
\end{equation}
Using Fej{\'e}r monotonicity of $\{x^{k}\}$ and \cite[Theorem 2.16
(ii)]{BB}, it suffices to show that every weak cluster point
$x^{*}$ of $\{x^{k}\}$ lies in $FixT$. To this end we assume that
$\{x^{n_{k}}\}$ is a subsequence of $\{x^{k}\}$ which converges
weakly to some point $x^{*} \in H$. In view of Remark
\ref{demi_closeness}, the operator $T -Id$ is demi-closed at $0$
if any of conditions ($i$), ($ii$) and ($iii$) is satisfied. Thus
the fact (\ref{demiclosedT}) and demi-closeness of $T-Id$ at $0$
implies that $x^{*}\in FixT$, which proves that $\{x^{k}\}$
converges weakly to some point in $FixT$.
\end{proof}

\begin{remark}
The following example shows that distinguishing between ($i$),
($ii$) and ($iii$) in Theorem \ref{theorem2} is necessary.
Let $U(x)=T_{2}T_{1}(x)$ where
\begin{equation*}\label{T1}
T_{1}(x) = \left\{ \begin{array}{ll}
\displaystyle 1 + \frac{1}{2}\left(\frac{1}{n}+\frac{1}{n+1}\right), &~~~~ x=1+\frac{1}{n}\vspace{0.1cm}\\
\frac{1}{2}x, &~~~~otherwise
\end{array} \right.
\end{equation*}
and
\begin{equation*}\label{T2}
T_{2}(x) = \left\{ \begin{array}{ll}
\displaystyle 1 + \frac{\sqrt{2}}{2}\left(\frac{1}{n}+\frac{1}{n+1}\right), &~~~~ x=1+\frac{\sqrt{2}}{n}\vspace{0.1cm}\\
\frac{1}{2}x, &~~~~otherwise.
\end{array} \right.
\end{equation*}

It is easy to check that $T_{1}$ and $T_{2}$ are strictly relaxed
cutter operators and consequently $U$ is a strictly relaxed cutter
operator. Furthermore, we have $Fix U=FixT_{1}\bigcap
FixT_{2}=\{0\}$. Assuming $x^{n}=1+\frac{1}{n},$ we get that
$\lim_{n\to \infty}\|T_{1}(x^{n})-x^{n}\|=0.$ Since $\lim_{n\to
\infty}x^{n}=1\notin FixT_{1},$ we conclude that $T_{1}-Id$ is not
demi-closed at $0.$ Similarly, considering
$x^{n}=1+\frac{\sqrt{2}}{n}$ gives that $T_{2}-Id$ is not
demi-closed at $0.$ Therefore, assuming only the case ($iii$) in
Theorem \ref{theorem2} does not guaranty the convergence results.
We next verify that $U-Id$ is demi-closed at $0$ which shows the
case ($ii$) is satisfied. We have
\begin{equation*}\label{T2T1}
U(x) -x = \left\{ \begin{array}{ll}
\displaystyle -\frac{1}{2} -\frac{3}{4n} +\frac{1}{4(n+1)}, &~~~~ x=1+\frac{1}{n}\vspace{0.1cm}\\
\displaystyle -1 - \frac{3\sqrt{2}}{4n}+\frac{\sqrt{2}}{2(n+1)}, &~~~~ x=2+\frac{2\sqrt{2}}{n}\vspace{0.1cm}\\
\displaystyle -\frac{3}{4}x, &~~~~otherwise
\end{array} \right.
\end{equation*}
and consequently $\lim_{n\to \infty}\|U(x^{n})-x^{n}\|=0$ if and
only if $\lim_{n\to \infty}x^{n}=0.$ Therefore, we get that $0 \in
Fix U$ and $U-Id$ is demi-closed at $0.$
\end{remark}

\subsection{Error reduction}\label{disc}
We now compare results of \cite[theorem 9]{CC}, \cite[Theorem
4.9.1]{C2013} and \cite[Theorem 14]{NM2014} with Theorem
\ref{theorem2}. Since the result of \cite[theorem 9]{CC} is based
on one string, we here assume $E=1,$ see Definition
\ref{tarifstring}. Let $e_{k}=\|x^{k}-z\|$ for $k\geq 0$ and $z\in
FixT$. Using (\ref{demiT}), we obtain
\begin{equation}\label{errorreduction}
e_k^2-e_{k+1}^2\geq
\frac{\lambda_{k}(2-\lambda_{k})}{m^2}\left(\frac{1}{\gamma}-\frac{1}{2}\right)^{2}\|T(x^k)-x^{k}\|^2
\end{equation}
which shows how much is big the difference of successive squared
errors. Indeed bigger right hand side in (\ref{errorreduction})
gives faster decay in error. Therefore, the minimum reduction of
error, assuming $\lambda_k=1$ in (\ref{errorreduction}), is
\begin{equation}\label{errorrelaxedcutter}
\frac{\left(\frac{1}{\gamma}-\frac{1}{2}\right)^{2}}{m^{2}}\|T(x^k)-x^{k}\|^2
\end{equation}
whereas this value in \cite[theorem9]{CC} is
\begin{equation}\label{errorcutter}
\frac{1}{4m^{2}}\|T(x^k)-x^{k}\|^2.
\end{equation}

Therefore, using $\alpha$-relaxed cutter operators, with proper
$\alpha,$ leads to have faster decay in error than using
$1$-relaxed cutter operators.

Setting $\{\alpha_{i}\}_{i=1}^m=1$ in Algorithm \ref{algt1}, i.e.
assuming all operators $T_{i\in I}$ are $1$-relaxed cutter, leads
to $\gamma=\max_{1\leq i\leq m}\alpha_i=1.$ Therefore, we obtain
that both lower bounds (\ref{errorrelaxedcutter}) and (\ref{errorcutter}) are equal.

\begin{remark}\label{rem6}
For simplicity assume that $\{\alpha_{i}\}_{i=1}^m=\alpha.$ Since
$\|T(x^k)-x^{k}\|^2$ converges to zero, we conclude that the
quantity (\ref{errorrelaxedcutter}) takes a large value if
$\alpha$ is a small number. Furthermore, using small value for
$\alpha$ diminishes the effect of each operator $T_i$ of
Definition \ref{tarifstring}. Therefore, we are not able to select
proper relaxation parameters only based on
(\ref{errorrelaxedcutter}).
\end{remark}

To compare Theorem \ref{theorem2} with extrapolated simultaneous
cutter operator which is analyzed  in \cite[Theorem 4.9.1]{C2013},
we assume $E=m$ and $m_{t}=1$ for $t=1,\ldots,E$ in Definition
\ref{tarifstring}. In this case, we have
\begin{equation}
 e_k^2-e_{k+1}^2\geq \lambda_{k}(2-\lambda_{k})\left(\frac{1}{\gamma}-\frac{1}{2}\right)^{2}\|T(x^k)-x^{k}\|^2.
\end{equation}
Assuming $\lambda_{k}=1$, the difference of successive squared
errors is
\begin{equation}\label{Sim_relax}
\displaystyle\left(\frac{1}{\gamma}-\frac{1}{2}\right)^{2}\|T(x^k)-x^{k}\|^2,
\end{equation}
whereas this value in \cite[Theorem 4.9.1]{C2013} is estimated by
\begin{equation}\label{SIM_cutter}
\|T(x^k)-x^{k}\|^2.
\end{equation}
Therefore, choosing proper $\alpha$-relaxed cutter operators leads
to have faster reduction of error than $1$-relaxed cutter
operators.

Comparing Theorem \ref{theorem2} with results in \cite{NM2014}, it
should be mentioned that the analysis in \cite[Theorem 14]{NM2014}
are based on sQNE operators and the relaxation parameters
$\lambda_k$ lie in $[\varepsilon,1-\varepsilon]$.  Furthermore,
their convergence analysis is valid if at least $T-Id$ or
$\{U_t-Id\}_{t=1}^E$ is demi-closed at zero whereas in Theorem
\ref{theorem2} the demi-closedness of $\{T_{i}-Id\}_{i=1}^m$ gives
the convergence result too. As it is seen in Theorem \ref{lemma3},
the step-size function depends on $\|y^{i}\|$ and $\|U_{t}(x)-x\|$
for $i=1,\ldots,m_{t}$ and $t=1,\ldots,E.$ It allows us to use
demi-closeness property of $T_{i}-Id$ and $U_{t}-Id$ in Theorem
\ref{theorem2}. However, the step-size function in \cite{NM2014}
is defined as
\begin{equation*}
\sigma_{max}(x)=\frac{\sum_{t=1}^{E}\omega_{t}\|U_{t}(x)-x\|^{2}}{\|T(x)-x\|^{2}}.
\end{equation*}
It only depends on the $\|U_{t}(x)-x\|$ for $t=1,\ldots,E$. Since
assuming the demi-closeness of $T_{i}-Id$ does not give any
convergence result, this property is not assumed in \cite[Theorem
14]{NM2014}. Therefore, to compare Theorem \ref{theorem2} with the
results in \cite{NM2014} in the equal conditions, we do not
consider demi-closeness of $T_{i}-Id$ for $i=1,\ldots,m$. In this
case, Theorem \ref{relaxedcutter} and Lemma \ref{lemma:gam} lead
to the following results
\begin{align*}
\left<z-x, T(x)-x\right>&=\sum_{t=1}^{E}\omega_{t}\left<z-x, U_{t}(x)-x\right>\\
&\geq\sum_{t=1}^{E}\frac{\omega_{t}}{\gamma_{m_{t}}}\|U_{t}(x)-x\|^{2}\\
&\geq\sum_{t=1}^{E}\frac{\omega_{t}}{\gamma_{m_{t}}}\|U_{t}(x)-x\|^{2}\\
&\geq\frac{1}{\bar{\gamma}}\|T(x)-x\|^{2},\\
\end{align*}
where $\bar{\gamma}=\max_{t=1,\ldots,E} \gamma_{m_{t}}$. Therefore
the step-size function in Theorem \ref{lemma3} can be rewritten as
\begin{equation}\label{New:sigma}
\sigma_{max}=\frac{\sum_{t=1}^{E}\frac{\omega_{t}}{\gamma_{m_{t}}}\|U_{t}(x)-x\|^{2}}{\|T(x)-x\|^{2}}
\end{equation}
which is bounded below by $\frac{1}{\bar{\gamma}}$. If we use
step-size defined in (\ref{New:sigma}) in Theorem \ref{theorem2},
we have
\begin{equation}
e_k^2-e_{k+1}^2\geq
\lambda_{k}(2-\lambda_{k})\frac{1}{\bar{\gamma}^{2}}\|T(x^k)-x^{k}\|^2.
\end{equation}
Assuming $\lambda_{k}=1$, the difference of successive squared
errors is
\begin{equation}\label{new:upper}
\frac{1}{\bar{\gamma}^{2}}\|T(x^k)-x^{k}\|^2,
\end{equation}
whereas this value in \cite[Theorem 14]{NM2014} is estimated
\begin{equation}\label{sQNE_upper}
 \frac{1}{4}\|T(x^k)-x^{k}\|^2.
\end{equation}
Since $\bar{\gamma}<2$, we conclude that for every relaxed cutter
operator we have faster reduction of error comparing with
\cite{NM2014}.

At the end, consider again $E=1$ and $m_{t}=m$ in Definition
\ref{tarifstring} and let $T_{i\in I}$ are $1$-relaxed cutter
operators. If we assume that  $T_{i}-Id$ are demi-closed at $0$
then the minimum reduction of error for generalized relaxation of
cutter operators is derived in (\ref{errorcutter}) whereas this
value without assuming the demi-closeness property of $T_{i}-Id$
is estimated in (\ref{new:upper}) and (\ref{sQNE_upper}) for
strictly relaxed cutter and sQNE operators, respectively.
Therefore the generalized relaxation of cutter operators was
inherently able to have faster reduction of error.
\subsection{Constraints}\label{const}
In this section we consider the projected version of Algorithm
\ref{algt1} for a general Hilbert space and for finite dimensional
Euclidean space, i.e.,  $\mathbb{R}^n.$

Let $\Omega$ be a closed convex subset of $H$ and $Fix
T_{\sigma,\lambda}\cap \Omega\ne \emptyset.$ Since $\sigma$ is far
from zero, see Theorem \ref{lemma3}, we conclude that $Fix
T_{\sigma,\lambda}=Fix T.$ Therefore we assume $Fix T\cap
\Omega\ne \emptyset.$  We now consider projected version of
Algorithm \ref{algt1} with constant relaxation parameter
$\lambda_k=\lambda$ as follows.
\begin{algorithm}\label{pro-algt1}
\emph{\\Initialization:} $x^{0}\in H$ is arbitrary. \newline
\emph{Iterative Step:} Given $x^{k},$ compute
\begin{eqnarray}
x^{k+1}=P_{\Omega}T_{\sigma,\lambda}(x^k)\nonumber
\end{eqnarray}
where $P_{\Omega}$ denotes the orthogonal projection onto
$\Omega.$
\end{algorithm}

Using Theorem \ref{lemma3} and Theorem \ref{relaxedcutter}, we
conclude that $P_{\Omega}T_{\sigma,\lambda}$ is an $\eta$-relaxed
cutter operator. Based on Theorem \ref{cuttertosqne},
$P_{\Omega}T_{\sigma,\lambda}$ is an $\frac{2-\eta}{\eta}$-SQNE.
On the other hand, based on \cite[Theorem 3.4.3]{C2013} any SQNE
operator is asymptotically regular, namely,
\begin{equation}\label{asym}
\lim_{k\rightarrow\infty}\|P_{\Omega}T_{\sigma,\lambda}x^k-x^k\|=0.
\end{equation}
If we assume that $P_{\Omega}T_{\sigma,\lambda}-Id$ is demi-closed
at $0$ then, using the similar arguments to those in the proof of
Theorem \ref{theorem2}, in particular, by \cite[Theorem 2.16
(ii)]{BB} and \cite[Theorem 2.1.26]{C2013}, the sequence
$\{x^{k}\}$ generated by Algorithm \ref{pro-algt1} converges
weakly to a point $x^{*}\in Fix
P_{\Omega}T_{\sigma,\lambda}=\Omega\cap FixT$. By applying Remark
\ref{demi_closeness} and in view of Remark \ref{rem2}, it is not
difficult to see that $P_{\Omega}T_{\sigma,\lambda}-Id$ is
demi-closed at $0$ if either $T-Id$, $U_{t}-Id$, $t=1,\ldots,E$ or
$T_{i}-Id$, $i=1,\ldots,m$ is demi-closed at $0$. We summarize the
above results as follows.
\begin{theorem}\label{propo1}
Let $\sigma=\sigma_{max}$ be step size function, $\bigcap_{i\in
I}FixT_{i}\neq \emptyset$ and
$\lambda\in[\varepsilon,2-\varepsilon]$ for an arbitrary constant
$\varepsilon\in(0,1)$. Assume that $\Omega\cap Fix T\ne\emptyset$.
The sequence generated by Algorithm \ref{pro-algt1} converges
weakly to a point in $Fix T\cap\Omega$, if one of the following
conditions is satisfied
\begin{itemize}
\item[($i$)] $T-Id$ is demi-closed at $0$, or
\item[($ii$)] $U_{t} -Id$ are demi-closed at $0$, for all $t=1, 2, \cdots, E$, or
\item[($iii$)] $T_{i} -Id$ are demi-closed at $0$, for all $i=1, 2, \cdots, m$.
\end{itemize}
\end{theorem}
We now consider the case $H=\mathbb{R}^n.$ We assume that every
operator $T_{i\in I}$, is continuous and strictly relaxed cutter.
In this case we don't need to explicitly assume that $T_{i}-Id$,
$U_{t}-Id$ or $T-Id$ is demi-closed at zero. Indeed, the
continuity of $T_{i\in I}$ implies that all of the above-mentioned
operators are demi-closed at zero. Thus, we have the following
corollary.

\begin{corollary}\label{propo2}
Let $\sigma=\sigma_{max}$ be step size function, $\bigcap_{i\in
I}FixT_{i}\neq \emptyset$  and
$\lambda\in[\varepsilon,2-\varepsilon]$ for an arbitrary constant
$\varepsilon\in(0,1)$. Assume that $\Omega\cap Fix T\ne\emptyset$
and $\{T_i\}_{i\in I}$ are continuous and strictly relaxed cutter
operators. The generated sequence of Algorithm \ref{pro-algt1},
where $H=\mathbb{R}^n$  converges to a point in $Fix T\cap\Omega.$
\end{corollary}
\begin{remark}\label{rem7}
The same convergence analysis as Corollary \ref{propo2} is still
true for Algorithm \ref{algt1} where $\{T_i\}_{i\in I}$ are
continuous and strictly relaxed cutter with $Fix T\ne\emptyset.$
\end{remark}

\section{Applications}\label{apl}
In this section we reintroduce two state-of-the-art iteration
methods which are based on strictly relaxed cutter operators.
First we begin with block iterative methods which are used for
solving linear systems of equations (inequalities) and later we
employ subgradient projections for solving nonlinear convex
feasibility problems. In all numerical tests, the case
$\alpha_t=1$ with various $\lambda_k$ where the number of string
is one, i.e. $E=1$, means that the only cutter operator, see
\cite{CC}, is used. The rest of pairs in all tables are our
results for various $\alpha_t$ and $\lambda_k.$ Also, for
simplicity, we assume $\alpha_t=\alpha$ and $\lambda_k=\lambda$ in
all tests where $t=1,\dots,p$ and $k\geq 0$. All the numerical
results are performed with Intel(R) Xenon(R) E5440 CPU 2.83 GHz,
8GB RAM, and the codes are written in Matlab R2013a.
\subsection{Block iterative methods}\label{sec:block}
 Let $A\in\mathbb{R}^{m\times n}$ and
$b\in\mathbb{R}^m$ are given. We assume the consistent linear
system of equations
\begin{equation}\label{block1ap}
Ax = b.
\end{equation}
Let $A$ and $b$ be partitioned into $p$ row-blocks
$\{A_t\}_{t=1}^p$ and $\{b^t\}_{t=1}^p,$ respectively. We now
consider the following operators
\begin{equation}\label{t8ap}
T_t(x)=x+\frac{\alpha_t}{\rho(A_t^T M_t A_t)} A_{t}^T
M_{t}(b^t-A_t x) \textrm{ for } t=1,\cdots,p
\end{equation}
where $\{\alpha_t\}_{t=1}^p$ and $\{M_t\}_{t=1}^p$ stand for
relaxation parameters and symmetric positive definite weight
matrices respectively. Also $\rho(B)$ denotes the spectral radius
of $B.$ If $E=1$ and $\sigma=1$ then Algorithm \ref{algt1}, with
operators (\ref{t8ap}), is called fully simultaneous method and
fully sequential method where $p = 1$ and $p = m,$ respectively.
With $M_t$ equal to the identity we get the classical Landweber
method \cite{landwe}. Other choices give rise to, e.g., Cimmino's
method \cite{Cimmino}, the CAV method \cite{CGG2001}, and, with a
componentwise scaling, the DROP algorithm \cite{CEHN2008}.

If $0<\varepsilon\leq \alpha_{t}\leq 2-\varepsilon \textrm{ for }
t=1,\dots, p,$ then \cite[lemmas 3 and 4]{NDH2012} gives that the
operator $T_t$ of (\ref{t8ap}) is not only strictly relaxed cutter
but also is nonexpansive. Note that the operator $T_t$ is an
$\alpha_t$-relaxed cutter. Since $\{T_t\}_{t=1}^p$ are
nonexpansive, we conclude that $\{T_t-Id\}_{t=1}^p$ are
demi-closed at zero, see Remark \ref{rem2}. Since composition and
convex combination of nonexpansive operators are nonexpansive, we
get that operators $U_t-Id$ and $T-Id$ are demi-closed at zero.
Therefore all conditions of Theorem \ref{theorem2} are satisfied.
However, based on Corollary \ref{propo2} and Remark \ref{rem7} we
do not need to verify the demi-closedness property of $T_t.$

\begin{remark}
It should be noted that a wide range of iterative methods, for
solving linear systems of equations, is based on strictly relaxed
cutter operators, see \cite[lemma 4]{NDH2012} and
\cite{CE2002,CEHN2008,CGG2001,DHC2009,EHL1981,E1980,EN2009,JW2003}.
Furthermore, such iterations appear in many applications which one
can find, for example, in signal processing, system theory,
computed tomography, proton computerized tomography and other
areas.
\end{remark}

Using (\ref{dec3}) the step size function (\ref{sigmamax}) can be
written as
\begin{equation}\label{smax1}
\sigma_{max}=\frac{\displaystyle\sum_{t=1}^{E}\omega_t\left(\displaystyle\sum_{i=1}^{m_{t}}\left(\frac{1}{\alpha_{i}}
-\frac{1}{2}\right)\|y^{i}\|^2+\frac{1}{2}\|\sum_{i=1}^{m_{t}}y^{i}\|^2\right)}{\|T(x)-x\|^2}
\end{equation}
where $x\in H\backslash Fix T.$ Therefore, using (\ref{t8ap}), we
have
\begin{equation}\label{smax2}
\sigma_{max}=\frac{\displaystyle\sum_{t=1}^{E}\omega_t\left(\displaystyle\sum_{i=1}^{m_{t}}\frac{\alpha_{i}(2-\alpha_{i})}
{\rho(A_{i}^{T}M_{i}A_{i})^2}\|A_{i}^{T}M_{i}(b_{i}-A_{i}u^{i-1})\|^2+\|U_{t}(x)-x\|^2\right)}{2\|\sum_{t=1}^{E}\omega_{t}U_{t}(x)-x\|^2}
\end{equation}
where $u^{i}$ is defined in (\ref{notation}).

\begin{remark}
As a special case of (\ref{t8ap}), we assume $p=m$ and
$M_{t}=\frac{1}{\|a^{t}\|^2}$ for $t=1,\cdots,m$ where $a^t$ and
$b^t$ show the $t$-row of $A$ and $b$ respectively. Therefore we
obtain $T_t$ as below
\begin{equation}\label{projection}
T_{t}(x)=x+\alpha_{t}\frac{b^{t}-\left<a^{t},x\right>}{\|a^{t}\|^2}a^{t}
\end{equation}
which is the orthogonal projection, where $\alpha_{t}=1,$ of $x$
onto hyperplane $\left\{x\in \mathbb{R}^n
|\left<a^{t},x\right>=b^{t}\right\}$ for $t=1,2, \cdots, m.$
Therefore, Algorithms \ref{algt1} and \ref{pro-algt1} are the
accelerated version of full sequential Kaczmarz's method and its
projected version, respectively.
\end{remark}

\subsection{Subgradient projection}\label{sec:sub}
The subgradient method uses the subgradient calculations instead
of orthogonal projections onto the individual sets for solving
convex feasibility problem. We will examine the subgradient
projection operator in Algorithms \ref{algt1} and \ref{pro-algt1}.
As it is mentioned before, one can use Algorithm \ref{pro-algt1}
based on Theorem \ref{propo1}.

Let $i\in J=\{1,2,\cdots, M\}$, the index set, and $g_i :D
\subseteq \mathbb{R}^{n} \rightarrow \mathbb{R}$ be convex
functions. We consider finding a solution $x^{*}\in D$ (assuming
its existence) of the following system of convex inequalities
\begin{equation}\label{funcap}
g_i(x)\leq 0,~~~\textrm{ for } i\in J.
\end{equation}
Let $g^{+}_{i}(x)= max \{ 0 , g_{i}(x)\}$, and denote the solution
set of (\ref{funcap}) by $S=\{x | g_{i}(x)\leq 0,~i\in J \}.$ Thus
$g^{+}_{i}(x)$ is a convex function and
\begin{equation}\label{moadelap}
S=\{x | g^{+}_{i}(x)= 0,~i\in J\}.
\end{equation}

Let $\ell_{i}(x)$ and $\partial g_{i}^{+}(x)$ denote subgradient
and set of all subgradients of $g_i$ at $x,$ respectively. Here a
vector $t\in\mathbb{R}^n$ is called subgradient of a convex
function $g$ at a point $y\in \mathbb{R}^n$ if $\langle t,
x-y\rangle \leq g(x)-g(y)$ for every $x\in \mathbb{R}^n.$ It is
known that the subgradient of a convex function always exist. We
consider the following operators which are used in cyclic
subgradient projection method, see \cite{cl81},
\begin{equation}\label{tt1}
T_{t}(x)=x-\alpha_{t}\frac{g_{t}^{+}(x)}{\|\ell_{t}(x)
\|^2}\ell_{t}(x)
\end{equation}
where $0<\varepsilon\leq \alpha_{t}\leq 2-\varepsilon.$ Clearly
$T_t$ is an $\alpha_t$-relaxed cutter. Based on analysis of
\cite[section 4]{CC}, see also \cite[theorem 4.2.7]{C2013},
$T_t-Id$ is demi-closed at $0$ under a mild condition and this
condition holds for any finite dimensional spaces. Therefore the
last condition of Theorem \ref{theorem2} is satisfied.

Similar to (\ref{smax2}), we obtain the following step size
function for the operator (\ref{tt1})
\begin{equation}\label{sigmamaxSubgradient}
\sigma_{max}=\frac{\displaystyle\sum_{t=1}^{E}\omega_t\left(\displaystyle\sum_{i=1}^{m_{t}}\alpha_{i}(2-\alpha_{i})\left(\frac{g_{i}^{+}(u^{i-1})}{\|\ell_{i}(u^{i-1})
\|}\right)^2+\|U_{t}(x)-x\|^2\right)}{2\|\sum_{t=1}^{E}\omega_{t}U_{t}(x)-x\|^2}.
\end{equation}

\begin{remark}\label{rem9}
Assuming $E=1$ and $\alpha_{i}=1$ for $i=1,\cdots,m$ give the
special case of Algorithm \ref{algt1}. If we use $T_t$ of
(\ref{tt1}) and the step size function (\ref{sigmamaxSubgradient})
then Algorithm \ref{algt1} leads to accelerated scheme of cyclic
subgradient projections method which was proposed in \cite{CC} and
\cite{NA2013}. 
\end{remark}

\begin{table}[h!]
\caption{ The results of $10$ quadratic examples using different
values of $\alpha_t$ and $\lambda_k$ where $E=1.$ The first and
the second parts of a pair indicate the average of iteration
numbers and the computational times (per second), respectively.}
\label{tab:3}
\begin{center}\tiny
\begin{tabular}{ccc|c|cc}
\hline
$\lambda_{\downarrow},~\alpha \rightarrow$ & $0.001$ & $0.5$ & $1$ & $1.5$ & $2-0.001$ \\
\toprule
 0.001 &(*,*) &(*,*)&(*,*)&(*,*)&(*,*) \\
 0.5 &(57,45.65) &(56,44.76)&(60,47.81)&(74,58.98)&(84,6680) \\\hline
 1 &(28,22.81) &(36,28.39)&(42,33.94)&(40,32.10)&(44,35.16)\\\hline
1.5 & (14,11.57) &(25,20.10)&(17,13.78)&(14,11.39)&(15,12.22) \\
2-0.001 &(19,15.47) &(25,20.08)&(22,17.71)&(19,15.29)&(15,12.17)\\
\bottomrule
\end{tabular}
\end{center}
\end{table}

We next examine $10$ nonlinear systems of inequalities with $500$
variables which are produced randomly. Each nonlinear system
consists of $100$ convex functions. We now explain how one of them
is produced. After generating the matrices
$G_i\in\mathbb{R}^{500\times 500}$ and the vectors
$c_i\in\mathbb{R}^{500}$ for $i=1,\cdots,100,$ we consider the
following convex functions
\begin{equation}\label{num1}
g_i(x) = x^T G_i^T G_i x + c_i^T x + d_i, ~i=1,\dots,100
\end{equation}
and calculate $d_i$ such that $g_i(y)\leq 0$ where
$y=(1,\cdots,1)^T.$ Therefore the solution set $S =\{ x|g_i(x)\leq
0,~i=1,\cdots, 100\}$ has at least one point. The components of
$G_i$ and $c_i$ lie in the interval $[-1,1].$ We generate randomly
a starting point, which its components lie in $[-10,10],$ for all
$10$ problems. Also we stop the iteration when $g^{+}_{i}(x^k)\leq
10^{-6}$ for all $i=1,\cdots, 100$ or $\|Tx^k-x^k\|\leq 10^{-16}$
or the number of iterations exceeds $k=1000.$ Also, we use equal
weights for all numerical tests.

In Table \ref{tab:3} we consider one string, i.e., $E=1.$ Also,
the first and the second parts of a pair indicate the average of
iteration numbers and the computational times (per second),
respectively. The sign $``*"$ means no feasible point is achieved
within our mentioned criteria. Based on Table \ref{tab:3}, the
best result of \cite{CC}, i.e., using $\alpha=1$ with different
$\lambda,$ is obtained by choosing $\alpha=1$ and $\lambda=1.5.$
However, based on our analysis which allows us to use
$\alpha\in(0,2),$ by setting  $\alpha=1.5$ and $\lambda=1.5$ we
reduce both, the number of iterations and the computational time.
Therefore, based on Table \ref{tab:3}, one may select a proper
value for $\alpha$ to reduce the number of iterations.

To see the effect of using generalized relaxation technique (grt),
we consider various number of strings, i.e. $E=2, 4, 5, 10, 20,$
and assume $\alpha=\lambda=1.$ As it is seen in Table \ref{tab:4},
using grt reduces the number of iterations and consequently the
computational times. On the other hand, the only increasing of the
number of strings does not guarantee to reduce the number of
iterations, see the last line of Table \ref{tab:4}. Comparing the
results of Table \ref{tab:4} emphasizes that using grt together
with many strings is able to reduce notably the number of
iterations .

\begin{table}[h!]
\caption{ The results of $10$ quadratic examples using different
sizes of strings where $\alpha=\lambda=1.$ The first and the
second parts of a pair indicate the average of iteration numbers
and the computational times (per second), respectively.}
\label{tab:4}
\begin{center}\tiny
\begin{tabular}{cccccc}
\hline
$E$ & $2$ & $4$ & $5$ & $10$ & $20$ \\
\toprule
 \textrm{with grt}&(36,22.36) &(31,14.50)&(36,15.65)&(28,10.77)&(28,9.68) \\
 \textrm{without grt}&(85,52.47) &(170,79.41)&(212,92.01)&(430,160.91)&(868,298.74)
\\
\bottomrule
\end{tabular}
\end{center}
\end{table}

\section{Conclusion}
In this paper we studied the generalized relaxation of string
averaging operator which is based on strictly relaxed cutter
operators. We showed that this operator is strictly relaxed cutter
operator by restricting the step size function. We analyzed a
fixed point iteration method with its projected version, based on
this operator. The capability of the method was examined by
employing state-of-the-art iterative methods. Our numerical tests
showed that one may choose relaxation parameters $\alpha\ne1$ to
reduce the number of iterations comparing with the case
$\alpha=1.$ The numerical results emphasize that using generalized
relaxation technique together with many strings is able to reduce
remarkably the number of iterations.
\section*{Acknowledgment}
\noindent We wish to thank an anonymous referee for helpful
suggestions which improved our paper.


\end{document}